\documentclass[11pt]{article}
\usepackage{amsfonts}
\usepackage{amsmath}
\usepackage{amsthm}
\usepackage{amscd}
\usepackage{amsbsy}
\usepackage{graphicx}
\usepackage{indentfirst, latexsym, bm,amssymb}
\usepackage{bbding}

\usepackage[pagewise]{lineno}
\usepackage[colorlinks=true]{hyperref}
\advance\textwidth by +1.0in \advance\textheight by +1.0in
\advance\oddsidemargin by -0.5in \advance\evensidemargin by -1.0in
\advance\topmargin by -0.5in
\newtheorem {theorem} {Theorem}
\newtheorem {proposition} [theorem]{Proposition}
\newtheorem {corollary} [theorem]{Corollary}
\newtheorem {lemma}  [theorem]{Lemma}

\newcommand{\R}{{\mathbb R}}

\usepackage{CJK}
\newtheorem{Defi}[theorem]{Definition}

\title{\large\bf On the mixing and Bernoulli properties for geodesic flows on rank $1$ manifolds without focal points}

\author{
\and Fei Liu
\thanks{College of Mathematics and System Science, Shandong University of Science and Technology, Qingdao, 266590, P.R. China.
e-mail: liufei$@$math.pku.edu.cn.}
\and Xiaokai Liu
\thanks{Department of Mathematics, Northwestern University, Evanston, IL 60208, USA.
e-mail: xkliu@math.northwestern.edu}
\and Fang Wang
\thanks{School of Mathematical Sciences, Capital Normal University, Beijing, 100048, China; and Beijing Center for Mathematics and Information
Interdisciplinary Sciences (BCMIIS), Beijing 100048, P.R. China. e-mail: fangwang@cnu.edu.cn.}}
\date{\today}

\begin{document}
\maketitle

\begin{abstract}

	If $(M,g)$ is a smooth compact rank $1$ Riemannian manifold without focal points, it is shown that the measure $\mu_{\max}$ of maximal entropy for the geodesic flow is unique. 
In this article, we study the statistic properties and prove that this unique measure $\mu_{\max}$ is mixing. 
Stronger conclusion that the geodesic flow on the unit tangent bundle $SM$ with respect to $\mu_{\max}$ is Bernoulli is acquired provided $M$ is a compact surface with genus greater than one and no focal points.
\\


\end{abstract}


\section{\bf Introduction}\label{intro}

\setcounter{section}{1}
\setcounter{equation}{0}\setcounter{theorem}{0}

The geodesic flow plays a significant role in modern theories of both differential geometry and dynamical systems, and, primarily on Riemannian manifolds, has been extensively studied (cf.~\cite{Kn2,Pa} for a comprehensive introduction). Moreover, geodesic flows on manifolds with negative/non-positive curvature are attractive subject nowadays with a bunch of works on the dynamics and ergodic theories done. Among most of these researches, the negative/non-positive curvature condition is critical, which contains deep information such as the convexity of distance function, the structure of fundamental groups, the geometry and topology of covering spaces, etc, and provides us rich research tools.

A natural generalization of the negatively/non-positively curved manifold is the manifold without focal points, which is similar to the manifolds of non-positive curvature on many aspects, but allows the existence of some subsets with positive curvature. By relaxing the curvature condition, some crucial properties, however, no longer hold, resulting in obstacles in studying geodesic flows on the manifolds. Benefiting from the latest progress in geometry and dynamics, we discovered a series of new properties on the dynamics of their geodesic flows on smooth compact manifolds without focal points recently.

In \cite{LW}, the entropy-expansiveness of geodesic flows on manifolds without focal points and surfaces without conjugate points, respectively is shown.
In \cite{LWW1}, it is discovered that on a smooth compact rank $1$ manifold without focal points the geodesic flow has a unique measure of maximal entropy.
Recently, in \cite{LZ}, with X.~Zhu, the first author proved that the geodesic flow on a compact rank $1$ manifold without focal points is transitive. All these works have built a solid foundation for our work on the mixing and Bernoulli properties of geodesic flows on compact rank $1$ Riemannian manifolds without focal points.

Let $(M,g)$ be a connected compact manifold with a complete Riemannian metric $g$. For each point $p \in M$ and unit tangent vector $v \in S_{p}M$, there is a unique geodesic $\gamma_{v}$ satisfying the initial conditions $\gamma_v(0)=p$ and $\gamma'_v(0)=v$. The geodesic flow $\phi=(\phi^{t})_{t\in\mathbb{R}}$ on the unit tangent bundle $SM$ is defined as:
\[
\phi^{t}: SM \rightarrow SM, \qquad (p,v) \mapsto
(\gamma_{v}(t),\gamma'_{v}(t)),\ \ \ \ \forall\ t\in \R .
\]
Throughout this paper, we use $\gamma_{v}\subset M$ to denote the geodesic curve generated by $v\in SM$, and $(\gamma_{v},\gamma'_{v})\subset SM$ to denote the corresponding trajectory of the geodesic flow on $SM$. We also use $\gamma_v'$ to denote the trajectory for notational simplicity.

A $\phi$-invariant probability measure $\mu$ is called the measure of maximal entropy if the measure-theoretic entropy $h_{\mu}(\phi)\geq h_{\nu}(\phi)$ for any other $\phi$-invariant probability measure $\nu$. From the variational principle, we know that $h_{\mu}(\phi)=h_{\text{top}}(g)$, where $h_{\text{top}}(g)$ is the topological entropy of the geodesic flow on $SM$.

The concept of (geometric) rank was first introduced by Ballmann, Brin and Eberlein in \cite{BBE}.
Given a unit vector $v \in SM$, we define the rank of the vector $\mathbf{rank(v)}$ to be the dimension of the space of parallel Jacobi fields along the geodesic curve $\gamma_{v}$,
and the rank of the manifold $\mathbf{rank(M)}:=\min\{\text{rank}(v) \mid v \in SM\}$.
From the definition, we can see $\text{rank}(\gamma'(t))=\text{rank}(\gamma'(0))$, $\forall\ t\in \mathbb{R}$, thus we can define the rank for a geodesic $\gamma$ as $\mathbf{rank(\gamma)}=\text{rank}(\gamma'(0))$. In particular, a rank $1$ geodesic is a geodesic with no parallel perpendicular Jacobi field, and a rank $1$ manifold is a Riemannian manifold admitting a rank $1$ geodesic.

On a rank $1$ manifold $M$, $SM$ splits into two invariant subsets:
the regular set $\mathbf{Reg}:= \{v\in SM \mid \text{rank}(v)=1\}$ and the singular set $\mathbf{Sing}:= SM \setminus \text{Reg}$. Here $\mbox{Reg}$ is an open subset of $SM$.
There has been a well-known conjecture since 1980's that the regular set has full Liouville volume given the compact rank $1$ manifold has non-positive curvature or no focal points. A positive answer to this conjecture will imply the ergodicity of the geodesic flow on such manifolds with respect to the Liouville measure. In \cite{LWW3}, with W. Wu, the first and third authors proved the conjecture with a mild condition and showed that the geodesic flow of a surface without focal points is ergodic with respect to the Liouville measure.

Next we introduce focal points and conjugate points. Let $\gamma$ be a geodesic on $(M,g)$. Two points $p=\gamma(t_{1})$ and $q=\gamma(t_{2})$ on $\gamma$ are called $\mathbf{focal}$ if there is a Jacobi field $J$ along $\gamma$ such that $J(t_{1})=0$, $J'(t_{1})\neq 0$ and $\frac{d}{dt}\| J(t)\|^{2}\mid_{t=t_{2}}=0$. We call $p=\gamma(t_{1})$ and $q=\gamma(t_{2})$ are  $\mathbf{conjugate}$ points if there is a non-identically zero Jacobi field $J$ along $\gamma$ with $J(t_{1})=J(t_{2})=0$. A compact Riemannian manifold $(M,g)$ is called a manifold
{\bf without focal points/without conjugate points} if there is no pair of focal points / conjugate points on any geodesic in $(M,g)$.
Obviously, a manifold without conjugate points is also a manifold without focal points. One can also check that a manifold with non-positive curvature has no focal points. In this sense, the concept of manifolds without focal / conjugate points is a natural and non-trivial generalization of the manifolds with non-positive curvature, because it can be constructed that a compact manifold without focal points whose curvature is not everywhere non-positive (cf.~\cite{Gu}).

In \cite{Kn1}, Knieper studied the dynamics of the geodesic flows on a compact rank $1$ manifold with non-positive curvature, and proved the existence, uniqueness, and ergodicity of the invariant measure of maximal entropy. This result answered Katok's famous conjecture on the existence and uniqueness of the measures of maximal entropy on the setting that the manifold is rank $1$ with non-positive curvature (cf.~\cite{BuKa}).  Recently, the first and third authors and W. Wu, generalized Knieper's work and proved the existence, uniqueness, and ergodicity of the measure of maximal entropy on a compact rank $1$ manifolds without focal points.

\begin{theorem}[cf. Liu-Wang-Wu \cite{LWW1}]\label{LWW}
Suppose $(M,g)$ is a smooth compact rank $1$ Riemannian manifold without focal points,
then the geodesic flow on $(M,g)$ has a unique measure of maximal entropy $\mu_{\max}$.
\end{theorem}

In this article, we study the statistical properties of $\mu_{\max}$ and show that it is mixing. Moreover, when $\dim(M)=2$, we prove that $\mu_{\max}$ is also Bernoulli. We organize this paper as following: our two main results will be strictly stated in Section \ref{mainresults}, and proved in Section \ref{mix} and Section \ref{proofBernoulli} respectively. In Section \ref{geometric}, we exhibits some important geometric properties of rank $1$ manifolds without focal points, which will be frequently used in our subsequent discussion. The measure $\mu_{\max}$ of maximal entropy will be studied afterwards in Section \ref{maximal}. In Section \ref{cratio}, the concept of cross ratio will be introduced and discussed.


\section{\bf{Main results}}\label{mainresults}
We present our two main results in this section. First of all, we concern the mixing property of $\mu_{\max}$. Mixing is a basic concept in the statistical theory of dynamical systems. Let $(X,\mathfrak{B})$ be a Borel space, and $(T^t)_{t\in\mathbb{R}}$ be a flow on $X$. An invariant probability measure $\mu$ of the flow $(T^t)_{t\in\mathbb{R}}$ is called mixing if for any two $\mu$ measurable sets $A, B\subset X$,
$$\lim_{t\rightarrow\infty}\mu(A\cap T^t(B))=\mu(A)\mu(B).$$  An equivalent condition of the mixing property is the decay of correlation: for all $f\in L^2(X, \mu)$ with $\int_X f d\mu=0$, $f\circ T^t$ converges to $0$ in the weak-$L^2$ topology (cf.~\cite{Wal}).

In this article, we prove the following theorem:
\begin{theorem}\label{mixing}
The maximal entropy measure $\mu_{\max}$ in Theorem \ref{LWW} is mixing.
\end{theorem}

This theorem  generalizes M.~Ballitot's classical result in \cite{B}, which states that Knieper's invariant measure of maximal entropy in \cite{Kn1} for the geodesic flow on a rank $1$ manifolds with non-positive curvature is mixing. A useful tool in proving our theorem is Babillot's Lemma 1 in \cite{B}. We cite this lemma here:

\begin{lemma}[cf. Babillot \cite{B}]\label{Babillot}
Suppose $(X,\mathfrak{B})$ is a Borel space,  $\mu$ is a probability measure on  $(X,\mathfrak{B})$, and $(T^t)_{t\in\mathbb{R}~\mbox{or}~\mathbb{Z}}$ is a dynamical system on $X$ preserving $\mu$. Let $f\in L^2(X,\mu)$ be a function with $0$ $\mu$-average: $\int fd\mu=0$. If $f\circ T^t$ does not converge to $0$ in the weak-$L^2$ topology, then there is a sequence $t_n\rightarrow\infty$ and a function $\psi\in L^2(X,\mu)$ which is not $\mu$-almost everywhere constant, such that $$f\circ T^{t_n}\rightarrow \psi~~\&~~f\circ T^{-t_n}\rightarrow\psi,~~n\rightarrow+\infty,$$ in the weak-$L^2$ topology.
\end{lemma}

A measure preserving system $(T, \mu)$ on $(X,\mathfrak{B})$ is called Bernoulli if it is conjugate to a Bernoulli shifting. A $\mu$-preserving flow $(T^t)_{t\in \mathbb{R}}$ is called Bernoulli if $(T^1, \mu)$ is Bernoulli.  It is easy to check that Bernoulli implies mixing. Our next theorem shows that when $\dim(M)=2$, the measure $\mu_{\max}$ in Theorem \ref{LWW} is Bernoulli. This is a extension of the result of Ledrappier, Lima and Sarig in \cite{LLS}.

\begin{theorem}\label{Bernoulli}
For a compact surface with genus greater than $1$ and without focal points, the geodesic flow is Bernoulli
with respect to the unique maximal entropy measure $\mu_{\max}$ in Theorem \ref{LWW}.
\end{theorem}


\section{\bf Some properties of rank $1$ manifolds without focal points}\label{geometric}
\setcounter{equation}{0}\setcounter{theorem}{0}

In this section, we present some geometric results on rank $1$ manifolds without focal points that will be used throughout the paper.

Let $(M,g)$ be a compact Riemannian manifold and $X$ be the universal covering manifold of $M$. Let $d$ be the distance function on $X$ induced by the lifted Riemannian metric $\tilde{g}$ on $X$. Suppose $h_{1}$ and $h_{2}$ are both geodesics in $X$. We call $h_{1}$ and $h_{2}$ are $\mathbf{positively ~asymptotic}$ if there is a positive number $C > 0$ such that
\begin{equation}\label{e1}
d(h_{1}(t),h_{2}(t)) \leq C, ~~\forall~ t \geq 0.
\end{equation}
Similarly, we say $h_{1}$ and $h_{2}$ are $\mathbf{negatively ~asymptotic}$ if \eqref{e1} holds for all $t \leq 0$.
$h_{1}$ and $h_{2}$ are said to be $\mathbf{biasymptotic}$ if they are both positively asymptotic and negatively asymptotic.
The positive / negative asymptoticity builds an equivalence relation on the geodesics on $X$. We denote the equivalent class positively / negatively asymptotic to a given geodesic $\gamma$ by $\gamma(+\infty)$ / $\gamma(-\infty)$ respectively,  and call these classes $\mathbf{points ~at ~infinity}$.
Obviously, we have $\gamma_{v}(-\infty)=\gamma_{-v}(+\infty)$. We use $X(\infty)$ to denote the set of all points at infinity,
and call it the $\mathbf{boundary~ at ~infinity}$.

Now we introduce some notations. Let $\overline{X}=X \cup X(\infty)$.
For each point $p \in X$ and $v\in S_{p}X$, each points $x, y \in \overline{X}-\{p\}$, positive numbers $\epsilon$ and $r$,
the following notations will be used:
\begin{itemize}
\item{} ~~$\gamma_{p,x}$ is the geodesic from $p$ to $x$ with $\gamma_{p,x}(0)=p$.
\item{} ~~$\measuredangle_{p}(x,y)=\measuredangle(\gamma'_{p,x}(0),\gamma'_{p,y}(0))$.
\item{} ~~$TC(v,\epsilon,r)= \{q \in \overline{X}\mid \measuredangle_{p}(\gamma_{v}(+\infty),q)< \epsilon\} - \{q \in \overline{X}\mid d(p,q)\leq r\}$.
\end{itemize}

The set $TC(v,\epsilon,r)\subset \overline{X}$ is called the $\mathbf{truncated ~cone}$ with axis $v$ and angle $\epsilon$. No hard to check $\gamma_{v}(+\infty) \in TC(v,\epsilon,r)$. The so called $\mathbf{cone ~topology}$ is the unique topology $\tau$ on $\overline{X}$ such that for each $\xi \in X(\infty)$
the set of truncated cones containing $\xi$ forms a local basis for $\tau$ at $\xi$. Under the cone topology, $\overline{X}$ is homeomorphic to the closed unit ball in $\mathbb{R}^{\text{dim}(X)}$,
and $X(\infty)$ is homeomorphic to the unit sphere $\mathbb{S}^{\text{dim}(X)-1}$.
For more details about the cone topology, see \cite{Eb1}.

The first proposition in this section is the property of \emph{the continuity at infinity}.
\begin{proposition}[cf. Liu-Wang-Wu \cite{LWW1}]\label{pro1}
 Let $X$ be a simply connected manifold without focal points, the following map is continuous.
 \begin{equation*}
	 \begin{aligned}
		 \Psi: SX \times [-\infty,+\infty] &\rightarrow \overline{X}\\
		 (v,t)\quad &\mapsto \gamma_{v}(t)
	 \end{aligned}
 \end{equation*}
\end{proposition}

The following result was first discovered by Ballmann in \cite{Ba}, and then extended by Watkins and Liu-Wang-Wu independently.

\begin{proposition}[cf. Watkins \cite{Wat}; Liu-Wang-Wu \cite{LWW1}]\label{pro6}
Let $X$ be a simply connected manifold without focal points and $v\in SX$ be a unit vector with $rank(v)=1$.
Then for any $\epsilon > 0$ and $a \neq 0$, there are neighborhoods $U_{\epsilon}$ of $\gamma_{v}(-\infty)$ and $V_{\epsilon}$ of $\gamma_{v}(+\infty)$ such that for each pair $(\xi,\eta)\in U_{\epsilon}\times V_{\varepsilon}$,
there is a rank $1$ geodesic $\gamma_{\xi,\eta}$ connecting $\xi$ and $\eta$ with $\gamma_{\xi,\eta}(-\infty)=\xi$ and $\gamma_{\xi,\eta}(+\infty)=\eta$, and $d(\gamma_{v}(t),\gamma_{\xi,\eta})<\epsilon$ for any $t$ between $0$ and $a$.
\end{proposition}

We state some useful corollary from \cite{LWW1} of the previous properties for further discussion here:
\begin{proposition}\label{corollary1}
Under the condition of Proposition \ref{pro6}, the following results hold:
\begin{enumerate}
\item For each $\xi\in X(\infty)$, there is a rank $1$ geodesic $\gamma_+$ with $\gamma_+(+\infty)=\xi$ and a rank $1$ geodesic $\gamma_{-}$ with $\gamma_{-}(-\infty)=\xi$.
\item Let $\gamma$ be a rank $1$ geodesic axis of some $\alpha \in \Gamma$.
For any neighborhood $U \subset \overline{X}$ of $\gamma(-\infty)$ and neighborhood $V \subset \overline{X}$ of $\gamma(+\infty)$,
there is a positive integer $N$ such that
$$\alpha^{n}(\overline{X}-U)\subset V, ~~\alpha^{-n}(\overline{X}-V)\subset U,~~\forall~n>N.$$
\item $\Gamma$ acts minimally on $X(\infty)$, i.e.~for any $\xi \in X(\infty)$, $\overline{\Gamma\xi}=X(\infty)$.
\end{enumerate}
\end{proposition}

Next we will introduce the concepts of Busemann function and horospheres. For each pair of points $(p,q)\in X \times X$ and a point at infinity $\xi \in X(\infty)$, the function
$$b_{p}(q,\xi):=\lim_{t\rightarrow +\infty}\{d(q,\gamma_{p,\xi}(t))-t\},$$ is called the $\mathbf{Busemann ~function}$ determined by $p, q$ and $\xi$.
Here $\gamma_{p,\xi}$ is the geodesic ray connecting $p$ to $\xi$.
The level sets of the Busemann function $b_{p}(q,\xi)$ are called the $\mathbf{horospheres  ~centered~ at ~\xi}$. By using the language of horosphere, we are able to define the stable/unstable manifold for a rank $1$ recurrent vector. We denote the horosphere centering at $\xi$ and passing through $\gamma_{p,\xi}(t)$ by $\mathbf{\mbox{Hor}_{\xi}(\gamma_{p,\xi}(t))}$.

Let $\pi: SX\rightarrow X$ be the standard projection sending a tangent vector to its base point. Define: $$\mbox{Hor}^+(v):=\mbox{Hor}_{\gamma_v(+\infty)}(\pi(v)),~~\mbox{and}~~\mbox{Hor}^-(v):=\mbox{Hor}_{\gamma_v(-\infty)}(\pi(v)), ~~\forall~v\in SX,$$
and the stable and unstable manifolds for $v$:
\begin{align*}
W^s(v):=\{w\in SM~|~\pi(w)\in \mbox{Hor}^+(v),~w\perp \mbox{Hor}^+(v),~\gamma_{w}(+\infty)=\gamma_{v}(+\infty)\},\\
W^u(v):=\{w\in SM~|~\pi(w)\in \mbox{Hor}^-(v),~w\perp \mbox{Hor}^-(v),~\gamma_{w}(-\infty)=\gamma_{v}(-\infty)\}.
\end{align*}

Note that $W^{s/u}(v)$ may not be the standard stable/unstable manifolds of $v$. Recall the Knieper metric which is equivalent to the standard Sasaki metric on $SM$, and the distance function is $$d_1(v,w):=\max_{1\leq t\leq1}d(\gamma_v(t),\gamma_w(t)),~v,w\in SM$$
One can check that $d_1(\phi^t(v),\phi^t(w))$ may not converges to $0$ when $t\rightarrow \pm\infty$ for $w\in W^{s/u}(v)$ respectively. The following lemma, however, tells us that if $v$ is a rank $1$ recurrent vector, we have $d_1(\gamma_v(t),\gamma_w(t))\rightarrow 0$  when $t\rightarrow \pm\infty$ for $w\in W^{s/u}(v)$ respectively (cf. \cite{LWW1} and Knieper's original result on rank $1$ manifolds of non-positive curvature \cite{Kn1}).

\begin{lemma}[cf.~\cite{LWW1}]\label{asymptotic}
If $v\in SM$ is a rank $1$ recurrent vector, then for all $w\in W^s(v)$, $$d_1(\phi^t(v),\phi^t(w))\rightarrow 0,~t\rightarrow+\infty.$$
\end{lemma}
Therefore, for rank $1$ recurrent vector $v\in SM$, $W^{s/u}(v)$ are the stable/unstable manifolds for $v$ respectively. For more details of the Busemann functions and horosphere, please refer to ~\cite{Eb1, Ru1}.

\section{\bf{Measures of maximal entropy on manifolds without focal points}}\label{maximal}

In this section, we recall unique maximal entropy measure $\mu_{\max}$ discussed in \cite{LWW1}. This result is a variation of Knieper's work of \cite{Kn1}. The main idea originates from Knieper's work \cite{Kn1} and was later extended to rank $1$ manifolds without focal points in \cite{LWW1}.

We construct the measure of maximal entropy based on the Busemann density:

\begin{Defi}\label{def3}
Suppose $X$ is the universal covering of $M$, where $M$ is a simply connected Riemannian manifold without conjugate points.
$\Gamma \subset \text{Iso}(X)$ is the discrete subgroup such that $M=X/\Gamma$. Given a constant $r>0$, a family of finite Borel measures $\{\mu_{p}\}_{p\in X}$ on $X(\infty)$ is called an $r$-dimensional Busemann density if
\begin{enumerate}
\item $\forall~p,\ q \in X$ and $\mu_{p}$-a.e. $\xi\in X(\infty)$,
$$\frac{d\mu_{q}}{d \mu_{p}}(\xi)=e^{-r \cdot b_{p}(q,\xi)}.$$
\item For any Borel set $A \subset X(\infty)$ and any $\alpha \in \Gamma$,
$$\mu_{\alpha p}(\alpha A) = \mu_{p}(A).$$
\end{enumerate}
\end{Defi}

The existence and uniqueness of the $h$-dimensional Busemann density are given in \cite{LWW1}, where $h=h_{top}(g)$ is the topological entropy of the geodesic flow on $M$. The existence is proven based on the construction of the Patterson-Sullivan measure $\{\mu_p\}_{p\in X}$ by using the Poincar\'e series, and the uniqueness is proven by showing that Patterson-Sullivan measure is exactly the unique $h$-dimensional Busemann density. Moreover, in the same article it is shown that $\mbox{Supp}(\mu_p)=X(\infty)$ for all $p\in X$.

Define a projection map $P:SX \rightarrow X(\infty) \times X(\infty)$ by $P(v)=(\gamma_{v}(-\infty),\gamma_{v}(+\infty)).$ Let $\mathcal{I}^{P}=P(SX)=\{P(v)\mid v \in SX\}$ stand for the subset of pairs in $X(\infty) \times X(\infty)$ that can be connected by a geodesic.

Fix a point $p\in X$, we define a $\Gamma$-invariant measure $\overline{\mu}$ on $\mathcal{I}^{P}$ in the following way:
\begin{equation}\label{current}
d \overline{\mu}(\xi,\eta) = e^{h\cdot \beta_{p}(\xi,\eta)}d\mu_{p}(\xi) d\mu_{p}(\eta),
\end{equation}
Here $\mu_{p}$ is the Patterson-Sullivan measure on $X(\infty)$ discussed in the above,
$\beta_{p}(\xi,\eta)=-\{b_{p}(q,\xi)+b_{p}(q,\eta)\}$ is the Gromov product,
and $q$ is an arbitrary point on the geodesic $\gamma$ connecting $\xi$ and $\eta$.
One can check that the function $\beta_{p}(\xi,\eta)$ does not depend on the choice of $\gamma$ or $q$.
In geometric language, the Gromov product $\beta_{p}(\xi,\eta)$ is the length of the part of the geodesic
$\gamma_{\xi,\eta}$ between the horospheres $\mbox{Hor}_{\xi}(p)$ and $\mbox{Hor}_{\eta}(p)$.

This $\Gamma$-invariant measure $\overline{\mu}$ on $\mathcal{I}^{P}$ induces a $\phi$-invariant measure $\mu$ on $SX$ by
\begin{equation}\label{e:def of max entropy measure}
\mu(A)=\int_{\mathcal{I}^{P}} \text{Vol}\{\pi(P^{-1}(\xi,\eta)\cap A)\}d \overline{\mu}(\xi,\eta),
\end{equation}
for any Borel set $A\subset SX$. Here $\pi : SX \rightarrow X$ is the standard projection map and
Vol is the induced volume form on $\pi(P^{-1}(\xi,\eta))$.
From the definition of $P$, we know that $P^{-1}(\xi,\eta)=\emptyset$ if there is no geodesic connecting $\xi$ and $\eta$. When there are more than one geodesics connecting $\xi$ and $\eta$ with one geodesic having rank $k\geq 1$, we know that all these geodesics have rank $k$ by the flat strip theorem, which implies that $P^{-1}(\xi,\eta)$ is the $k$-flat sub-manifold connecting $\xi$ and $\eta$ consisting all rank $k$ geodesics between these two points. Particularly in the case $k=1$, $P^{-1}(\xi,\eta)$ is exactly the rank 1 (thus unique) geodesic connecting $\xi$ and $\eta$.

Following the above discussion we can conclude that for any Borel set $A\subset SX$ and $t\in \mathbb{R}$,
$\text{Vol}\{\pi(P^{-1}(\xi,\eta)\cap \phi^{t}A)\}=\text{Vol}\{\pi(P^{-1}(\xi,\eta)\cap A)\}$.
Therefore $\mu$ is both $\Gamma$-invariant and $\phi$-invariant.

The $\Gamma$-invariance implies that $\mu$ can be projected to a finite $\phi$-invariant measure on $SM$.
We may use $\mu$ to denote this projected measure later for notation simplicity. Note that $\mu$ is a finite measure, we can normalized it to a probability measure. Thus we can assume $\mu$ is a probability measure throughout the paper. Furthermore in \cite{LWW1}, we show that $\mu$ is the unique maximal invariant measure of the geodesic flow and denote it by $\mu_{\max}$.

The measure $\overline{\mu}$ is usually called the geodesic current associated to $\mu_{\max}$ (cf.~\cite{B}). Since $\mu_{\max}$ is supported on the regular set $\mbox{Reg}$ (cf.~\cite{LWW1}), the geodesic current $\overline{\mu}$ is a $\Gamma$-invariant ergodic measure supported
on $$\mathcal{R}:=\{(\xi,\eta) \in X(\infty) \times X(\infty)~|~ (\xi, \eta)=P(v),~\mbox{for~some~} v\in \mbox{Reg}\}.$$ So $\mathcal{R}$ is a $\overline{\mu}$-fully measure set in $X(\infty) \times X(\infty)$.
Moreover, from the expression of the geodesic current $\overline{\mu}$ (see (\ref{current})), we know that it is equivalent to $\mu_{p} \otimes \mu_{p}$, where $\mu_{p}$ is the Patterson-Sullivan measure for some $p\in X$,
which has full support on $X(\infty)$.

\section{\bf{The cross ratio}}\label{cratio}

For any $\xi, \eta \in X(\infty)$, let $\gamma_{\xi,\eta}$ be the connecting geodesic with
$\gamma_{\xi,\eta}(-\infty)=\xi$ and $\gamma_{\xi,\eta}(+\infty)=\eta$.
For $\xi, \eta, \xi', \eta' \in X(\infty)$, two rank $1$ connecting geodesics $\gamma_{\xi,\eta}$ and $\gamma_{\xi',\eta'}$ are called a $\mathbf{quadrilateral ~Quad(\xi, \eta, \xi', \eta')}$, if there exist two rank $1$ connecting geodesics $\gamma_{\xi,\eta'}$ and $\gamma_{\xi',\eta}$.
Proposition \ref{pro6} implies that the set $\mathcal{Q}$ of quadrilaterals is an open neighborhood of the diagonal $\mathcal{R} \times \mathcal{R}$, where $\mathcal{R}$ is the space of oriented non-parameterized rank $1$ geodesics.

\begin{proposition}\label{corss-ratio}
Let $X$ be a simply rank $1$ manifold without focal points.
For a quadrilateral Quad$(\xi, \eta, \xi', \eta')$, let $\{p_{n}\}^{\infty}_{n=1}$, $\{q_{n}\}^{\infty}_{n=1}$,
$\{p'_{n}\}^{\infty}_{n=1}$ and $\{q'_{n}\}^{\infty}_{n=1}$ be sequences of points in $X$ that converging to $\xi, \eta, \xi'$ and $\eta'$ respectively.
Define
$$Cr(p_{n}, q_{n}, p'_{n}, q'_{n})=d(p_{n},q_{n})+d(p'_{n},q'_{n})-d(p_{n},q'_{n})-d(p'_{n},q_{n}).$$
One can check $\lim_{n\rightarrow \infty}Cr(p_{n}, q_{n}, p'_{n}, q'_{n})$ exists and is independent of the choices of sequences,
which allows us to define $Cr(\xi, \eta, \xi', \eta')$ as this limit.
\end{proposition}
\begin{proof}
Denote the four horospheres centered at $\xi, \eta, \xi'$ and $\eta'$ by $\mbox{Hor}_{\xi}$, $\mbox{Hor}_{\eta}$, $\mbox{Hor}_{\xi'}$ and $\mbox{Hor}_{\eta'}$, and the corresponding horoballs by $\overline{\mbox{Hor}}_{\xi}$, $\overline{\mbox{Hor}}_{\eta}$, $\overline{\mbox{Hor}}_{\xi'}$ and $\overline{\mbox{Hor}}_{\eta'}$ respectively.
Let $X_{0}=X - \{\overline{\mbox{Hor}}_{\xi}\cup \overline{\mbox{Hor}}_{\eta}\cup \overline{\mbox{Hor}}_{\xi'} \cup \overline{\mbox{Hor}}_{\eta'}\}$. Choose $\mbox{Hor}_{\xi}$, $\mbox{Hor}_{\eta}$, $\mbox{Hor}_{\xi'}$ and $\mbox{Hor}_{\eta'}$ to be pairwise disjoint, thus
$$\gamma_{\xi,\eta}\cap X_{0}, ~~\gamma_{\xi',\eta'}\cap X_{0},  ~~\gamma_{\xi,\eta'}\cap X_{0},  ~~\gamma_{\xi',\eta}\cap X_{0}$$
consists of four geodesic segments with the lengths $d_{1}$, $d_{2}$, $d_{3}$ and $d_{4}$, respectively. We want to show that the limit above is equal to $d_{1} + d_{2}- d_{3} - d_{4}$.

First, for any two geodesics $\gamma_{1}$ and $\gamma_{2}$ with $\gamma_{1}(+\infty)=\gamma_{2}(+\infty):=\zeta$,
and any two horospheres $\mbox{Hor}^{1}_{\zeta}$ and $\mbox{Hor}^{2}_{\zeta}$ centered at $\zeta$, using the fact that the $\mbox{Hor}^{1}_{\zeta}$ and $\mbox{Hor}^{2}_{\zeta}$ are the level sets of Busemann function $b_{\cdot}(\cdot,\zeta)$, we can see that the lengths of the segments of $\gamma_{1}$ and $\gamma_{2}$ between two horospheres are equal, implying that the number $d_{1} + d_{2}- d_{3} - d_{4}$ is independent of the choice of the horospheres.

Next we will show that the four bi-infinite connecting geodesics $\gamma_{p_{n},q_{n}}$, $\gamma_{p'_{n},q'_{n}}$, $\gamma_{p_{n},q'_{n}}$ and $\gamma_{p'_{n},q_{n}}$
converge to the geodesics $\gamma_{\xi,\eta}$, $\gamma_{\xi',\eta'}$, $\gamma_{\xi,\eta'}$ and $\gamma_{\xi',\eta}$, respectively, with the topology induced by the Hausdorff metric of compact sets. For this purpose, we fix a point $p$ in $\gamma_{\xi,\eta}$ and parametrize this geodesic by setting $p=\gamma_{\xi,\eta}(0)$ . Since $p_{n} \rightarrow \xi$,
passing to a subsequence if necessary, we can assume that for each $n\in \mathbb{N}$, $p_{n} \in TC(-\gamma'_{\xi,\eta}(0),\frac{1}{n},n)$,
where $TC(-\gamma'_{\xi,\eta}(0),\frac{1}{n},n)$ is the truncated cone. That is to say that the angle between the two geodesics
$\gamma_{p,\xi}$ and $\gamma_{p,p_{n}}$ at $p$ is smaller than $\frac{1}{n}$ while the distance between $p$ and $p_{n}$ is greater than $n$.
If $\gamma_{p_{n},q_{n}}$ does not converge to $\gamma_{\xi,\eta}$, passing to a subsequence if necessary,
we can find at least one point $x$ in the geodesic $\gamma_{p_{n},q_{n}}$ and a number $r > 0$ such that
$$d(x,\gamma_{p_{n},q_{n}}) \geq r,  \quad n \in \mathbb{N}.$$
Without loss of generality, we can assume that $x=p$. We consider the following two cases:

\vspace{1ex}
$\bullet$ \textbf{Case I}~~$\{d(p,\gamma_{p_{n},q_{n}})\}^{\infty}_{n=1}$ is bounded from above by some constant $R$:
$$r \leq d(p,\gamma_{p_{n},q_{n}})\leq R, ~~~n \in \mathbb{N}.$$

We choose the parametrization of $\gamma_{p_{n},q_{n}}$ such that
$d(p,\gamma_{p_{n},q_{n}}(0))=d(p,\gamma_{p_{n},q_{n}})$. Passing to a subsequence if necessary,
we can assume $v=\lim_{n\rightarrow +\infty}\gamma'_{p_{n},q_{n}}(0)$.
Then $v\neq \gamma'_{p,\xi}(t)$ for all $t \in \mathbb{R}$. By Proposition \ref{pro1}, we have
$$\gamma_{v}(+\infty)=\gamma_{p,\xi}(+\infty)=\xi, ~~\gamma_{v}(-\infty)=\gamma_{p,\xi}(-\infty)=\eta.$$
Thus by the famous flat strip theorem (the no focal points version, cf. ~\cite{Os}), $\gamma_v$ and $\gamma_{p,\xi}$ bound a flat strip,
implying the rank of the geodesic $\gamma_{p,\xi}$ is greater than $1$, which contradicts to the assumption that $\gamma_{\xi,\eta}$ is a rank one geodesic.

\vspace{1ex}
$\bullet$ \textbf{Case II}~~$\{d(p,\gamma_{p_{n},q_{n}})\}^{\infty}_{n=1}$ is unbounded.

By the parametrization $p=\gamma_{\xi,\eta}(0)$, the two geodesics $\gamma_{\xi,\eta}$ and $\gamma_{p,\xi}$ actually coincide. We can choose $x_{n}, y_{n}\in \gamma_{\xi,\eta}$ for each $n\in \mathbb{N}$ such that
$$d(p_{n},x_{n})=d(p_{n},\gamma_{\xi,\eta}), ~~d(q_{n},y_{n})=d(q_{n},\gamma_{\xi,\eta}).$$
Let $b_{n}:[0,1]\rightarrow X$ be a smooth curve connecting $x_{n}$ and $p_{n}$ with $b_{n}(0)=x_{n}$ and $b_{n}(1)=p_{n}$,
and $\measuredangle_{p}(\xi,b_{n}(s))$ increases as $s$ growing up.
Similarly, let $c_{n}:[0,1]\rightarrow X$ be a smooth curve connecting $y_{n}$ and $q_{n}$ with $c_{n}(0)=y_{n}$ and $c_{n}(1)=q_{k}$, and $\measuredangle_{p}(\eta,c_{n}(s))$ increases as $s$ growing up. Let $\gamma_{n,s}$ be the unique geodesic connecting $c_{n}(s)$ and $b_{n}(s)$ for each $s \in [0,1]$, with the parametrization such that $\gamma_{n,0}(0)=p$ and $\gamma_{n,s}(0)$ is a smooth curve with respect to $s$.

We can find some $s_n\in (0,1]$ such that $d(p,\gamma_{n,s_{n}}(0))=r >0$ because $d(p,\gamma_{n}(0))=d(p,\gamma_{n,1}(0))\geq r$ and $d(p,p)=d(p,\gamma_{n,0}(0))=0$. Without loss of generality we can assume that $\lim_{n\rightarrow +\infty}\gamma'_{n,s_{n}}(0)= v_{\infty}\in SX$. We have $d(p,\pi(v_{\infty}))=r>0$.
Then it's easy to show that
$$d(b_{n}(s_{n}),p)\rightarrow +\infty, ~~d(c_{n}(s_{n}),p)\rightarrow +\infty.$$
Then by the fact
$$\max\{\measuredangle_{p}(\xi,b_{n}(s_{n})),\ \measuredangle_{p}(\eta,c_{n}(s_{n}))\}\leq \frac{1}{n}.$$
we get
$$\lim_{n\rightarrow +\infty}b_{n}(s_{n})=\xi,~~\lim_{n\rightarrow +\infty}c_{n}(s_{n})=\eta.$$
Proposition \ref{pro1} leads to
$$\gamma_{v_{\infty}}(+\infty)=\lim_{n\rightarrow +\infty}\gamma_{n,s_{n}}(+\infty)=\lim_{n\rightarrow +\infty}b_{n}(s_{n})=\xi,$$
$$\gamma_{v_{\infty}}(-\infty)=\lim_{n\rightarrow +\infty}\gamma_{n,s_{n}}(-\infty)=\lim_{n\rightarrow +\infty}c_{n}(s_{n})=\eta.$$
Given $X$ is a manifold without focal points, $b_{n}(s_{n})$ and $c_{n}(s_{n})$ cannot be on geodesic $\gamma_{\xi,\eta}$ simultaneously.
Thus for all $t\in \mathbb{R}$,  $v_{\infty}\neq \gamma_{\xi,\eta}'(t)$. We can conclude that $\gamma_v$ and $\gamma_{v_{\infty}}$ bound a flat strip, contradicting the assumption that $\gamma_{\xi,\eta}$ is a rank $1$ geodesic.

We have proved that the four sequences of geodesics $\gamma_{p_{n},q_{n}}$, $\gamma_{p'_{n},q'_{n}}$, $\gamma_{p_{n},q'_{n}}$ and $\gamma_{p'_{n},q_{n}}$ converge to the geodesics $\gamma_{\xi,\eta}$, $\gamma_{\xi',\eta'}$, $\gamma_{\xi,\eta'}$ and $\gamma_{\xi',\eta}$, respectively. Thus the infinite end points of these geodesics $\gamma_{p_{n},q_{n}}$, $\gamma_{p'_{n},q'_{n}}$, $\gamma_{p_{n},q'_{n}}$ and $\gamma_{p'_{n},q_{n}}$ also converge to the corresponding end points of geodesics $\gamma_{\xi,\eta}$, $\gamma_{\xi',\eta'}$, $\gamma_{\xi,\eta'}$ and $\gamma_{\xi',\eta}$.
From Proposition \ref{pro6}, we know that
the four segments of $\gamma_{p_{n},q_{n}}$, $\gamma_{p'_{n},q'_{n}}$, $\gamma_{p_{n},q'_{n}}$ and $\gamma_{p'_{n},q_{n}}$ intersected with $X_{0}$
converge to the four segments of $\gamma_{\xi,\eta}$, $\gamma_{\xi',\eta'}$, $\gamma_{\xi,\eta'}$ and $\gamma_{\xi',\eta}$ intersected with $X_{0}$, respectively. Write $d^{n}_{1}$, $d^{n}_{2}$, $d^{n}_{3}$ and $d^{n}_{4}$ to be the lengths of the four segments of $\gamma_{p_{n},q_{n}}$, $\gamma_{p'_{n},q'_{n}}$, $\gamma_{p_{n},q'_{n}}$ and $\gamma_{p'_{n},q_{n}}$ intersected with $X_{0}$, the only thing left we need to show is that $\lim_{n\to\infty} d^n_i=d_i,~i=1,2,3,4$.

Let $A_{n}$ (resp. $A$) be the intersection point of geodesic $\gamma_{p_n,q_n}$ (resp. $\gamma_{\xi,\eta}$) with the horosphere $\mbox{Hor}_{\xi}$. On the manifolds without focal points, the horospheres $\mbox{Hor}_\xi$ depends continuously on the center $\xi$ (cf.~Ruggiero \cite{Ru1}), implying $A_{n}\to A$. Thus, $d^n_1\to d_1$, so with the other three convergences.
\end{proof}

We call this limit the $\mathbf{cross ~ratio ~of ~the ~quadrilateral}$ and denoted the quantity by
$\mathbf{Cr(\xi,\eta,\xi',\eta')}$. The continuity of the cross ratio follows the continuity of the horospheres with respect to
their centers:

\begin{corollary}\label{continuous}
Let $X$ be a simply rank $1$ manifold without focal points, then the cross ratio is a continuous function on $\mathcal{Q}$.
\end{corollary}

By the triangle inequality, the following proposition is straightforward:

\begin{proposition}\label{positive}
	The cross ratio is strictly positive if the quadrilateral underneath defined by two intersecting rank $1$ geodesics.
\end{proposition}

The following pictures give a geometric explanation of the cross ratio. Consider a quadrilateral $\mbox{Quad}(\xi, \eta, \xi', \eta')$. Connect each pair $(\xi,\eta), (\xi, \eta'), (\xi', \eta), (\xi', \eta')$ by rank $1$ geodesics, denoted by $\gamma_{\xi,\eta}, \gamma_{\xi, \eta'}, \gamma_{\xi', \eta}, \gamma_{\xi', \eta'}$, respectively.
Take a tangent vector $v_0$ on $\gamma_{\xi,\eta}$ pointing to the positive direction of $\gamma_{\xi,\eta}$. Let $v_1 = W^s(v_0)\cap(\gamma_{\xi', \eta}(t),\gamma'_{\xi', \eta}(t))_{t\in\mathbb{R}}$. Obviously the foot point of $v_1$ is the intersection of $\gamma_{\xi', \eta}$ and $\mbox{Hor}^+(v_0)$.   Similarly, take $v_2=W^u(v_1)\cap(\gamma_{\xi', \eta'}(t),\gamma'_{\xi', \eta'}(t))_{t\in\mathbb{R}}$, whose foot point is exactly the intersection point of $\mbox{Hor}^-(v_1)$ and $\gamma_{\xi',\eta'}$. Take $v_3=W^s(v_2)\cap(\gamma_{\xi, \eta'}(t),\gamma'_{\xi, \eta'}(t))_{t\in\mathbb{R}}$ and $v_4=W^u(v_3)\cap(\gamma_{\xi, \eta}(t),\gamma'_{\xi, \eta}(t))_{t\in\mathbb{R}}$. Therefore $v_4=\phi^{\tau}(v_0)$ for some $\tau\in\mathbb{R}$. One can check that $\tau$, indicated as the green segment, is independent of the choice of $v_0$ and is exactly the cross ratio $Cr(\xi,\eta,\xi',\eta')$ (cf.~\cite{B}).
	\begin{figure}[hb]
		\centering
		\includegraphics[width=0.8\textwidth]{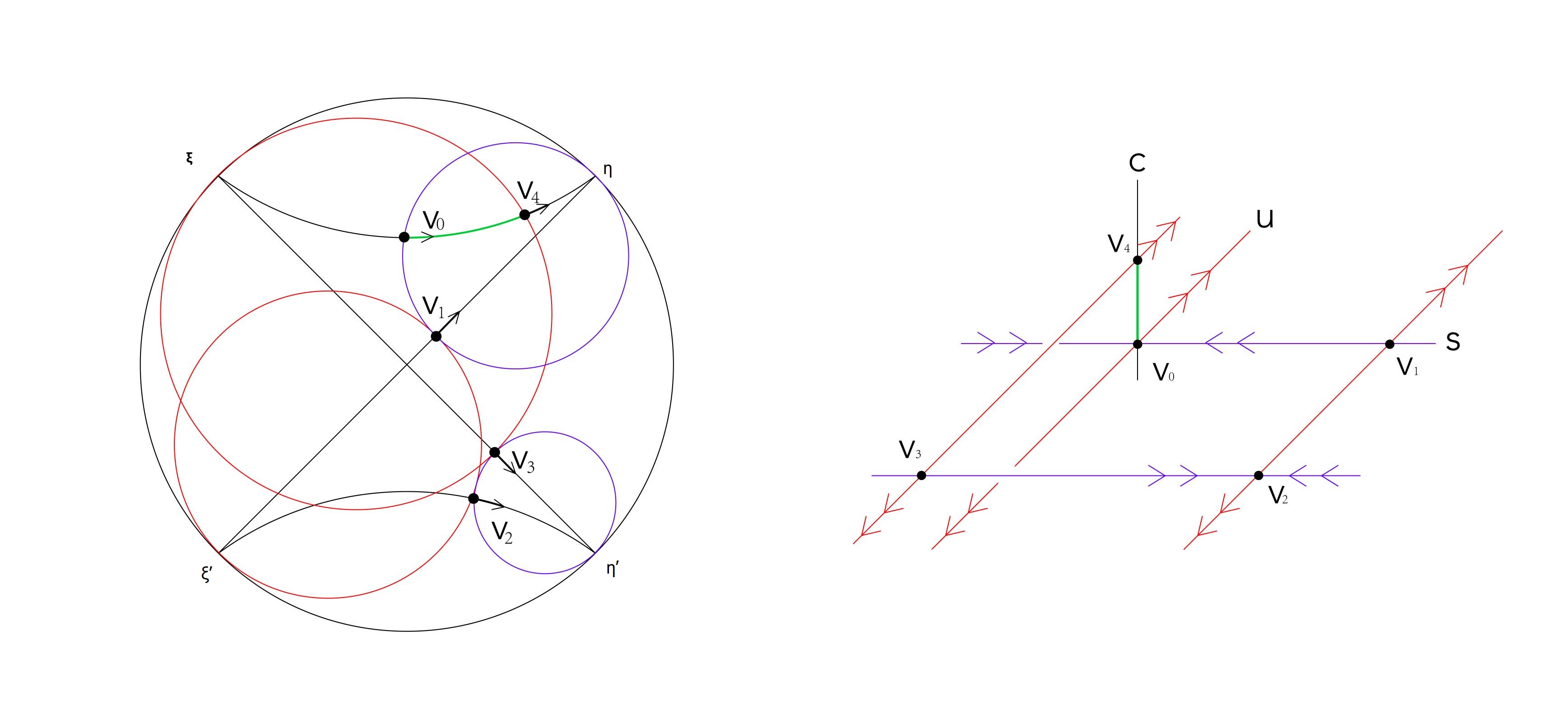}
	\end{figure}

Consider a bounded measurable function $f: SX\rightarrow \mathbb{R}$.
We say $f$ is $\mathbf{W^s ~invariant}$ if there is a full $\mu_{\max}$ measure subset $U\subset SX$, such that for all $v,w\in U$, $w\in W^s(v)$, $f(v)=f(w)$. Similarly we can define the  $W^u$ invariant functions. A $\mathbf{W^{s,u} ~invariant}$ function is both $W^u$ and $W^s$ invariant. The following lemma states that if $f$ is $W^{s,u}$ invariant and continuous on $\overline{\mu}$-almost every trajectory, then it is a periodic function on $\overline{\mu}$-almost every trajectory.

\begin{lemma}\label{period}
If a measurable function $f: SX\rightarrow \mathbb{R}$ is  $W^{s,u}$ invariant and continuous on $\overline{\mu}$-almost every trajectory,
then $f$ is periodic on $\overline{\mu}$-almost every trajectory of the geodesic flow.
\end{lemma}
\begin{proof}

	Let $U\subset SX$ be the set of rank $1$ vectors such that $\forall v\in U$, $f(v)=f(w)$ for all $w\in W^{u}(v)$ or $w\in W^{s}(v)$. Given $f$ is $W^{s,u}$ invariant, the following set is a $\overline{\mu}$-full measure set in $X(\infty) \times X(\infty)$: $$E=\{(\xi,\eta)\in \mathcal{R}~|~\gamma'_{\xi,\eta}(t)\in U,~ \forall~t\in \mathbb{R}\}.$$

For each $(\xi,\eta)\in E$, take $\xi',\eta'\in X(\infty)$ such that $(\xi, \eta'), (\xi', \eta), (\xi', \eta')\in E$. By the product structure of $\overline{\mu}$, the set of such elements $(\xi,\eta)\in E$ is a $\overline{\mu}$-full measure set. Let  $v_0$ be a tangent vector on $\gamma_{\xi,\eta}$ pointing to the positive direction of $\gamma_{\xi,\eta}$ (i.e. $v_0\in (\gamma'_{\xi,\eta}(t))_{t\in\mathbb{R}}$). Apply the procedure of the geometric construction of the cross ratio above, we can see $f(v_1)=f(v_0)$ as  $v_1 \in W^s(v_0)\cap E$. Similarly, we have $$f(v_4)=f(v_3)=f(v_2)=f(v_1).$$

Therefore $f(v_0)=f(v_4)$, which implies that  $f(v)=f(\phi^{\tau}(v))$ where $\tau=Cr(\xi,\eta,\xi',\eta')$, for all $v\in (\gamma'_{\xi,\eta}(t))_{t\in\mathbb{R}}$. Thus, the cross ratio $Cr(\xi,\eta,\xi',\eta')$ is a period of the function $t\rightarrow f(\gamma'_{\xi,\eta}(t))$. It follows that $f$ is periodic on almost all trajectories.
\end{proof}

\section{\bf{The proof of the mixing property}}\label{mix}
In this section, we give the proof of Theorem \ref{mixing}. This proof is inspired by M.~Babillot's work in \cite{B} and G.~Knieper's work in \cite{Kn1}.

\noindent\emph{Proof of Theorem \ref{mixing}}.
As we showed in \cite{LWW1}, the geodesic flow is ergodic with respect to the unique maximal entropy measure $\mu_{\max}$, with the regular set $\mbox{Reg}$ being a full measure set. We will prove $\mu_{\max}$ is mixing by contradiction.

Assume $\mu_{\max}$ is not mixing, there exists a continuous function $f: SM\rightarrow \mathbb{R}$ with $\int f d\mu_{\max}=0$ but $f \circ \phi^{t}$ doesn't converge to $0$ in the weak $L^{2}$-topology (cf.~\cite{B}).
By Lemma \ref{Babillot}, there is a function $\psi\in L^2(SM,\mu_{\max})$ which is not $\mu_{\max}$-almost everywhere constant, and a sequence numbers $\{s_{n}\}$ approaching to positive infinity such that under the weak $L^{2}$-topology:
$f \circ \phi^{s_{n}} \rightarrow \psi~$ and $f\circ \phi^{-s_{n}} \rightarrow \psi$.

Then, by the Banach-Saks Theorem, there are sub-sequences $\{s_{n_j}\}$ and $\{s'_{n_j}\}$ approaching to positive infinity, such that the Cesaro averages $$\psi^+_k(v):=\frac{1}{k}\sum_{j=1}^k f(\phi^{s_{n_j}})~~~\&~~~\psi^-_k(v):=\frac{1}{k}\sum_{j=1}^k f(\phi^{-s'_{n_j}})$$ approach to $\phi$, $\mu_{\max}$ almost surely.

Let $\psi^+=\limsup_{k\rightarrow \infty}\psi^+_k$ and $\psi^-=\limsup_{k\rightarrow \infty}\psi^-_k$. We have $$\psi(v)=\psi^+(v)=\psi^-(v),~~\mbox{for}~\mu_{\max}~\mbox{a.e.}~v\in SM.$$ Moreover, $\psi^+(u)=\psi^+(v)$ when $u,v\in SM$ are positively asymptotic; and $\psi^-(u)=\psi^-(v)$ when $u,v$ are negatively asymptotic.

Now let $\tilde{\psi}$, $\tilde{\psi}^+$ and $\tilde{\psi}^-$ on $SX$ denote the lifting of $\psi$, $\psi^+$ and $\psi^-$. We can see the above properties also hold for the lifted functions. Note that $\psi$ cannot be almost everywhere constant on almost all trajectories, otherwise by the ergodicity of $\mu_{\max}$, $\psi$ is in fact almost everywhere constant, contradicting to our assumption. Neither is $\tilde{\psi}$.
	
Smoothen $\tilde{\psi}$ by considering $\tilde{\psi}^*(v):=\int_0^{\varepsilon} \tilde{\psi}(\phi^t) dt$ for some $\varepsilon>0$. The function $\tilde{\psi}^*$ is continuous along trajectories. In addition, by taking $\varepsilon>0$ small enough, we can make sure that $\tilde{\psi}^*$ is not constant on almost all trajectories. We apply this procedure to $\tilde{\psi}^+$ and $\tilde{\psi}^-$ with the same $\varepsilon$ too to get $\tilde{\psi}^{*+}$ and $\tilde{\psi}^{*-}$ respectively. Then the following properties hold:
\begin{enumerate}
\item $\tilde{\psi}^*$, $\tilde{\psi}^{*+}$ and $\tilde{\psi}^{*-}$ are $\Gamma$-invariant functions on $SX$.
\item $\tilde{\psi}^*(v)=\tilde{\psi}^{*+}(v)=\tilde{\psi}^{*-}(v)$ for $\mu_{\max}$ almost every $v\in SX$;
\item If $u,v\in SX$ are positively asymptotic, then $\tilde{\psi}^{*+}(u)=\tilde{\psi}^{*+}(v)$. If $u,v\in SX$ are negatively asymptotic, then $\tilde{\psi}^{*-}(u)=\tilde{\psi}^{*-}(v)$.
\end{enumerate}
Here as we have indicated before, we use $\mu_{\max}$ to denote the lifting of $\mu_{\max}$ on $SX$. Since $\tilde{\psi}^*$, $\tilde{\psi}^{*+}$ and $\tilde{\psi}^{*-}$ are continuous on each trajectory, $\tilde{\psi}^*=\tilde{\psi}^{*+}=\tilde{\psi}^{*-}$ on almost all trajectories.

Now, take an arbitrary rank $1$ axis $c$ of some non-elementary $\alpha\in\Gamma$. By Proposition \ref{pro6}, we have an open neighborhood $U\subset X(\infty)$ of $c(-\infty)$ and an open neighborhood $V\subset X(\infty)$ of $c(+\infty)$, such that for any $\xi\in U$ and $\eta\in V$ there is a unique rank $1$ geodesic $\gamma_{\xi,\eta}$ with $\gamma_{\xi,\eta}(-\infty)=\xi$ and $\gamma_{\xi,\eta}(+\infty)=\eta$. Fix $U$ and $V$, let
$$\mathcal{G}(U,V)=\{c \mid c~ \mbox{is a geodesic with} ~c(-\infty)\in U ~\mbox{and} ~c(+\infty)\in V\},$$
$$\mathcal{G}_{rec}(U,V)=\{c\in \mathcal{G}(U,V)~\mid ~c' ~\mbox{is recurrent}\},$$
and
$$\tilde{\mathcal{G}}_{rec}(U,V)=\{c\in \mathcal{G}_{rec}(U,V)~\mid  ~\tilde{\psi}^*(v)=\tilde{\psi}^{*+}(v)=\tilde{\psi}^{*-}(v)~\mbox{on}~c\}.$$
Here we say $v\in SX$ is recurrent if $v$ is a lifting vector of a recurrent vector in $SM$ under the geodesic flow (or equivalently, there exist sequences $s_n\rightarrow +\infty$ and $\{\alpha_n\}\subset\Gamma$ with $\alpha_n(\phi^{s_n}(v))\rightarrow v$ as $n\rightarrow +\infty$).

Similarly, let
$$\mathcal{G}'(U,V)=\{c'(t) ~|~ c\in \mathcal{G}(U,V),~t\in\mathbb{R} \},$$
$$\mathcal{G}'_{rec}(U,V)=\{c'(t) ~|~ c\in \mathcal{G}_{rec}(U,V),~t\in\mathbb{R} \}$$
and
$$\tilde{\mathcal{G}}'_{rec}(U,V)=\{c'(t) ~|~ c\in \tilde{\mathcal{G}}_{rec}(U,V),~t\in\mathbb{R} \}.$$

By Poincar\'e recurrence theorem, we know $$\mu_{\max}(\mathcal{G}'(U,V)\setminus \mathcal{G}'_{rec}(U,V))=0.$$ Also, as we discussed earlier, we have $$\mu_{\max}(\mathcal{G}'(U,V)\setminus \tilde{\mathcal{G}}'_{rec}(U,V))=0.$$

\begin{lemma}\label{lem1}
For $\mu_p$ almost all $\xi\in U$ the set $$G_{\xi}:=\{\eta\in V~\mid~\exists~c\in\tilde{\mathcal{G}}_{rec}(U,V)~\mbox{with}~\eta=c(-\infty)\}$$ is a $\mu_p$-full measure set in $V$. Here $\mu_p$ is the Patterson-Sullivan measure on $X(\infty)$.
\end{lemma}
\begin{proof}
Let $\mathcal{E}'=\mathcal{G}'_{rec}(U,V)\setminus\mathcal{\widetilde{G}}'_{rec}(U,V)$, we have $\mu_{\max}(\mathcal{E}')=0$. For each $\xi\in U$, denote $G^{co}_{\xi}= V\setminus G_{\xi}$.
From the construction of $\mu_{\max}$ (see (\ref{e:def of max entropy measure})), we have
\begin{eqnarray*}
0 = \mu(\mathcal{E}')
& = & \int_{P(\mathcal{E}')}\text{Vol}(\pi\circ P^{-1}(\xi,\eta)\cap \mathcal{E}')e^{h\beta_{p}(\xi,\eta)}d\mu_{p}(\xi) d\mu_{p}(\eta)\\
& \geq & \int_{\xi\in U,~\eta\in G^{co}_{\xi}}d\mu_{p}(\xi) d\mu_{p}(\eta)=\int_{\xi\in U}\left(\int_{G^{co}_{\xi}}d\mu_{p}(\eta)\right)d\mu_{p}(\xi).
\end{eqnarray*}
Therefore, $$\int_{\xi\in U}(\int_{G^{co}_{\xi}}d\mu_{p}(\eta))d\mu_{p}(\xi)=0.$$

This implies that for  $\mu_{p}$-a.e.~$\xi \in U$, $\mu_{p}(G^{co}_{\xi})=0$. Thus, $\mu_{p}(G_{\xi})=\mu_{p}(V)$ for $\mu_{p}$-a.e.~$\xi \in U$.
\end{proof}

Based on lemma \ref{lem1}, we have
\begin{lemma}\label{lem2}
$\tilde{\psi}^*$ is periodic on almost all trajectories in $\tilde{\mathcal{G}}'_{rec}(U,V)$.
\end{lemma}
\begin{proof}
Take a geodesic $c\in\tilde{\mathcal{G}}_{rec}(U,V)$ with $c(-\infty)=\xi$ and $c(+\infty)=\eta$ with $\mu_p(V\setminus G_{\xi})=0.$ By Lemma \ref{lem1}, $\mu_p$-almost all $c\in\tilde{\mathcal{G}}_{rec}(U,V)$ fulfills this requirement. Take arbitrary $\xi',\eta'\in X(\infty)$ such that $$\gamma_{\xi, \eta'}, \gamma_{\xi', \eta}, \gamma_{\xi', \eta'}\in \tilde{\mathcal{G}}_{rec}(U,V).$$ The existence of $\xi', \eta'$ are also guaranteed by Lemma \ref{lem1}.

Let $v_0$ be a tangent vector on $\gamma_{\xi,\eta}$ pointing to the positive direction of $\gamma_{\xi,\eta}$ (i.e. $v_0\in (\gamma'_{\xi,\eta}(t))_{t\in\mathbb{R}}$). Recall the procedure of the geometric construction of the cross ratio, as reasoned in Lemma \ref{period}, we  have that $$\tilde{\psi}^*(v_1)=\tilde{\psi}^{*+}(v_1)=\tilde{\psi}^{*+}(v_0)=\tilde{\psi}^*(v_0),$$ since  $v_1 \in W^s(v_0)$ and $v_0$ is a recurrent vector.

Similarly, we have $$\tilde{\psi}^*(v_4)=\tilde{\psi}^*(v_3)=\tilde{\psi}^*(v_2)=\tilde{\psi}^*(v_1).$$
Therefore $\tilde{\psi}^*(v_0)=\tilde{\psi}^*(v_4)$, leading to  $\tilde{\psi}^*(v)=\tilde{\psi}^*(\phi^{T}(v))$ where $T=Cr(\xi,\eta,\xi',\eta')$ is the cross ratio, for all $v\in (c'(t))_{t\in\mathbb{R}}$. Thus $Cr(\xi,\eta,\xi',\eta')$ is a period of the function $t\mapsto \tilde{\psi}^*(c'(t))$ on the trajectory $(c,c')$, implying that $\tilde{\psi}^*$ is periodic on it.
\end{proof}

Note that in this proof, we showed that in fact the cross ratio is a period of $\tilde{\psi}^*$ on the corresponding trajectory. This observation plays a key point in our next argument.

\begin{lemma}\label{lem3}
$\tilde{\psi}^*$ is constant on almost all trajectories in $\tilde{\mathcal{G}}'_{rec}(U,V)$.
\end{lemma}
\begin{proof}
Take a geodesic $c\in\tilde{\mathcal{G}}_{rec}(U,V)$ with $c(-\infty)=\xi$ and $c(+\infty)=\eta$ such that $\mu_p(V\setminus G_{\xi})=0$. This property holds for $\mu_p$-almost all $c\in\tilde{\mathcal{G}}_{rec}(U,V)$. From the definition, we have $Cr(\xi,\xi,\eta,\eta)=0.$ By Lemma \ref{lem1} and the construction of $\mu_p$, we can find sequences $\xi_n\rightarrow \xi\in U$ and $\eta_n\rightarrow \eta \in V$ with $\gamma_{\xi_n,\eta_n}, \gamma_{\xi,\eta_n}, \gamma_{\xi_n,\eta} \in\tilde{\mathcal{G}}_{rec}(U,V)$. Moreover, we can also require that each geodesic $\gamma_{\xi_n,\eta_n}$ intersects $c$.

Proposition \ref{positive} implies that $Cr(\xi,\xi_n,\eta,\eta_n)$ is strictly positive for all $n\in\mathbb{Z}^+$. Using the continuity of the cross ratio (Corollary \ref{continuous}), we can conclude
$$Cr(\xi,\xi_n,\eta,\eta_n)\rightarrow Cr(\xi,\xi,\eta,\eta)=0~~\mbox{as}~n\rightarrow+\infty.$$

Thus these cross ratios as the positive period of $\tilde{\psi}^*$ on $(c,c')$ can be arbitrarily small, implying $\tilde{\psi}^*$ is constant on $(c,c')$ from the continuity of $\tilde{\psi}^*$ on this trajectory.
\end{proof}

Come back to the proof that $\mu_{\max}$ is mixing. Let $$\widetilde{V}=\{\eta\in V~\mid~\exists~c\in\tilde{\mathcal{G}}_{rec}(U,V),~c(+\infty)=\eta,~\mu_p(G_{c(-\infty)})=\mu_p(V)\}.$$

One can check that $\mu_p(\widetilde{V})=\mu_p(V)$. And if $c, \gamma\in\tilde{\mathcal{G}}_{rec}(U,V)$ are two geodesics with $c(+\infty)$=$\gamma(+\infty)$, $\tilde{\psi}^*$ is constant on $(c,c')$ implies it is also constant on  $(\gamma,\gamma')$, and vice verse.

Let $$Y=\bigcup_{\alpha\in\Gamma}\alpha(\widetilde{V}).$$ By Proposition \ref{corollary1}, $\mu_p(X(\infty)\setminus Y)=0$. So $Z:=\{c'(t)~\mid~c(-\infty)\in Y,~t\in\mathbb{R}\}$ is a $\mu_{\max}$ full measure set in $SX$. This implies that $\tilde{\psi}^*$ is constant on almost all trajectories in $SX$. By the $\Gamma$-invariance of $\tilde{\psi}^*$ and the ergodicity of $\phi^t$ with respect to $\mu_{\max}$, $\tilde{\psi}^*$ is constant $\mu_{\max}$-almost everywhere on $SX$, which contradicts to our assumption.

To summarize, the geodesic flow $\phi^t$ is mixing with respect to $\mu_{\max}$  the measure of maximal entropy. We complete the proof of Theorem \ref{mixing}.

\section{\bf{The Bernoulli property}}\label{proofBernoulli}

In this section, we give a brief proof to our Theorem \ref{Bernoulli}.

\noindent\emph{Proof of Theorem \ref{Bernoulli}}.
Given a compact surface, the sectional curvature at each point is the Gauss curvature. The Gauss-Bonnet formula suggests a region with all sectional curvatures be negative in it, provided that the surface has no focal points and its genus is greater than $1$. In the region, the surface admits rank $1$ geodesics, thus is a rank $1$ surface. We can say the geodesic flows on it has a unique maximal entropy measure $\mu_{\max}$ from theorem \ref{LWW}.

On the other hand, we know that the fundamental group of a compact surface with genus greater than $1$ has exponential growth rate.
The famous Dinaburg Theorem (cf. \cite{Di}) ensures that the topological entropy of geodesic flow is positive, as well as the maximal measure entropy. In fact, in \cite{LWW2}, we generalized Dinaburg Theorem from geodesic flows to more general autonomous Lagrangian systems.
And in \cite{LWW1}, we show that $\mu_{\max}$ is an ergodic measure with full mass on the rank $1$ set (also called the regular set).
Then by Theorem 1.2 of \cite{LLS}, the geodesic flow is Bernoulli with respect to the unique maximal entropy measure.
This wraps up the proof of Theorem \ref{Bernoulli}.

\section*{\textbf{Acknowledgements}}
F.~Liu is partially supported by NSFC under Grant Nos.11301305 and 11571207. F.~Wang is partially supported by NSFC under Grant No.11871045 and by the State Scholarship Fund from China Scholarship Council (CSC).

X.~Liu is under the supervision of Prof.~Zhihong Jeff Xia at Department of Mathematics,
Northwestern University, Evanston, IL, and would like to take this chance to acknowledge the constructive
advices and helps from Prof. Xia, as well as the great research environment and hospitality
offered by SUSTech with gratitude.


\end{document}